\documentclass[11pt]{amsart}
\usepackage[margin=1in]{geometry}

\usepackage{amssymb}
\usepackage{amsthm}
\usepackage{amsmath}
\usepackage{mathrsfs}
\usepackage{amsbsy}
\usepackage[all]{xy}
\usepackage{bm}
\usepackage{hyperref}
\usepackage{tikz}
\usepackage{array}
\usepackage{float}
\usepackage{enumerate}
\usepackage{xcolor}
\usepackage{hhline}
\allowdisplaybreaks
\usepackage[noadjust]{cite}

\usepackage{comment}

\usepackage[noabbrev,capitalise,nameinlink]{cleveref}

\definecolor{lavender}{rgb}{0.4,0,1.0}

\hypersetup{colorlinks=true, citecolor=lavender, linkcolor=lavender,urlcolor=lavender}

\usepackage[singlelinecheck=false]{caption}
\usepackage{tabu}
\usepackage{diagbox}

\definecolor{SockYellow}{rgb}{1,0.85,0}
\definecolor{SockGreen}{rgb}{0,0.5,0}

\newenvironment{enumerate*}%
  {\begin{enumerate}[(I)]%
    \setlength{\itemsep}{10pt}%
    \setlength{\parskip}{0pt}}%
  {\end{enumerate}}

\newtheorem{theorem}{Theorem}[section]
\newtheorem{proposition}[theorem]{Proposition}

\newtheorem{question}[theorem]{Question}

\newtheorem{lemma}[theorem]{Lemma}

\theoremstyle{definition}

\newtheorem{remark}[theorem]{Remark}

\newcommand{\foot}{\mathsf{foot}}

\newcommand{\Av}{\mathrm{Av}}

\newcommand{\sockr}{\includegraphics[height=.8cm]{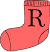}}
\newcommand{\sockb}{\includegraphics[height=.8cm]{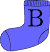}}
\newcommand{\sockg}{\includegraphics[height=.8cm]{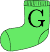}}
\newcommand{\socky}{\includegraphics[height=.8cm]{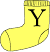}}

\newcommand{\ssockr}{\includegraphics[height=.6cm]{SockR}}
\newcommand{\ssockb}{\includegraphics[height=.6cm]{SockB}}
\newcommand{\ssockg}{\includegraphics[height=.6cm]{SockG}}
\newcommand{\ssocky}{\includegraphics[height=.6cm]{SockY}}
\newcommand{\ssockp}{\includegraphics[height=.6cm]{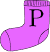}}

\newcommand{\dfn}[1]{\textcolor{blue}{\emph{#1}}}

\definecolor{MyGreen}{rgb}{0,0.86,0}

\begin{document}

\title[]{Foot-sorting for socks}

\author[Colin Defant and Noah Kravitz]{Colin Defant}
\address[]{Department of Mathematics, Massachusetts Institute of Technology, Cambridge, MA 02139, USA}
\email{colindefant@gmail.com}
\author[]{Noah Kravitz}
\address[]{Department of Mathematics, Princeton University, Princeton, NJ 08540, USA}
\email{nkravitz@princeton.edu}

\maketitle

\begin{abstract}
If your socks come out of the laundry all mixed up, how should you sort them?  We introduce and study a novel \emph{foot-sorting} algorithm that uses feet to attempt to sort a sock ordering; one can view this algorithm as an analogue of Knuth's stack-sorting algorithm for set partitions. The sock orderings that can be sorted using a fixed number of feet are characterized by Klazar's notion of set partition pattern containment. We give an enumeration involving Fibonacci numbers for the $1$-foot-sortable sock orderings within a naturally-arising class.  We also prove that if you have socks of $n$ different colors, then you can always sort them using at most $\left\lceil\log_2(n)\right\rceil$ feet, and we use a Ramsey-theoretic argument to show that this bound is tight. 
\end{abstract}

\section{Introduction}
\vspace{-0.2cm}
\begin{figure}[h]
\begin{center}\includegraphics[width=0.3\linewidth]{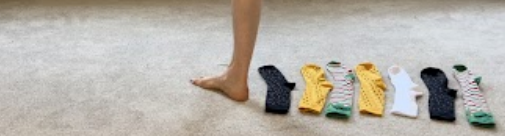}
\includegraphics[width=0.3\linewidth]{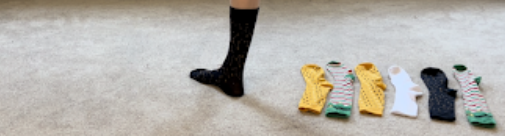}
\includegraphics[width=0.3\linewidth]{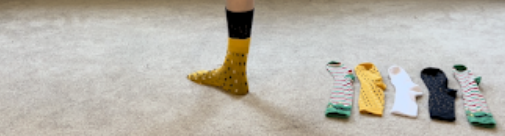} \\ \vspace{0.2cm}
\includegraphics[width=0.3\linewidth]{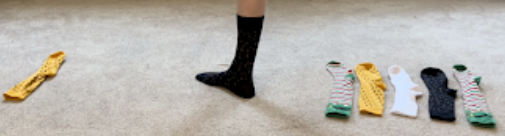}
\includegraphics[width=0.3\linewidth]{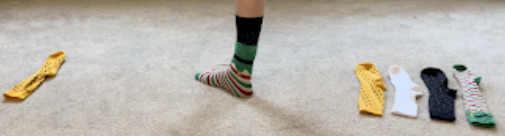}
\includegraphics[width=0.3\linewidth]{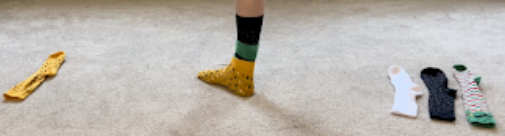} \\ \vspace{0.2cm}
\includegraphics[width=0.3\linewidth]{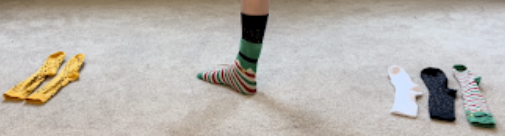}
\includegraphics[width=0.3\linewidth]{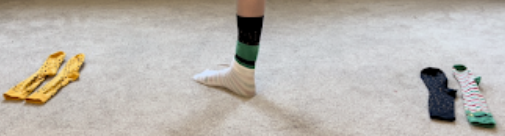}
\includegraphics[width=0.3\linewidth]{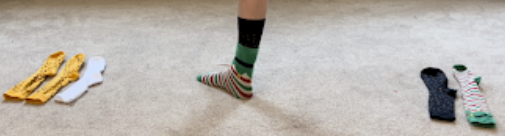} \\ \vspace{0.2cm}
\includegraphics[width=0.3\linewidth]{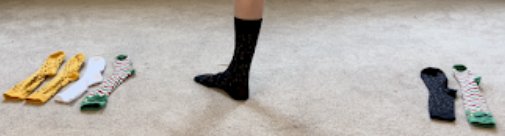}
\includegraphics[width=0.3\linewidth]{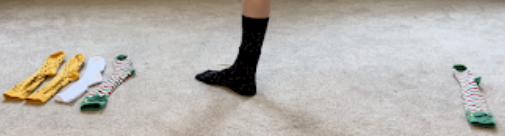}
\includegraphics[width=0.3\linewidth]{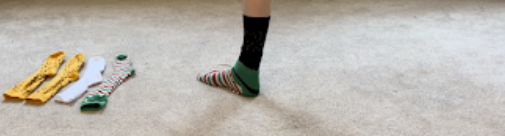} \\ \vspace{0.2cm}
\includegraphics[width=0.3\linewidth]{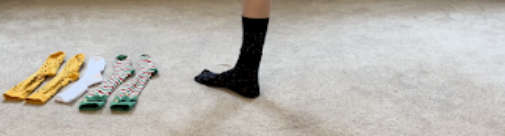}
\includegraphics[width=0.3\linewidth]{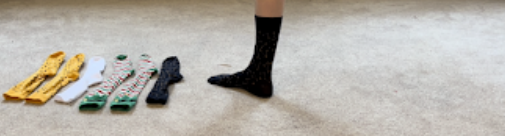}
\includegraphics[width=0.3\linewidth]{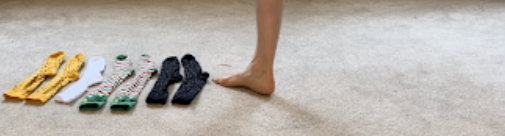}
  \end{center}
  \caption{Seven socks begin in the (unsorted) sock ordering \begin{center}\textsf{\textbf{BLACK {\color{SockYellow}YELLOW} {\color{SockGreen}P}{\color{red}E}{\color{SockGreen}P}{\color{red}P}{\color{SockGreen}E}{\color{red}R}{\color{SockGreen}M}{\color{red}I}{\color{SockGreen}N}{\color{red}T} {\color{SockYellow}YELLOW} {\includegraphics[height=0.2688cm]{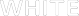}} BLACK {\color{SockGreen}P}{\color{red}E}{\color{SockGreen}P}{\color{red}P}{\color{SockGreen}E}{\color{red}R}{\color{SockGreen}M}{\color{red}I}{\color{SockGreen}N}{\color{red}T}}}.\end{center}  A real person uses the foot-sorting algorithm to sort these socks into \begin{center}\textsf{\textbf{{\color{SockYellow}YELLOW} {\color{SockYellow}YELLOW} {\includegraphics[height=0.2688cm]{WHITE}} {\color{SockGreen}P}{\color{red}E}{\color{SockGreen}P}{\color{red}P}{\color{SockGreen}E}{\color{red}R}{\color{SockGreen}M}{\color{red}I}{\color{SockGreen}N}{\color{red}T}  {\color{SockGreen}P}{\color{red}E}{\color{SockGreen}P}{\color{red}P}{\color{SockGreen}E}{\color{red}R}{\color{SockGreen}M}{\color{red}I}{\color{SockGreen}N}{\color{red}T} BLACK BLACK}}.\end{center} The pictures should be read row by row.}\label{Fig2}
\end{figure}
\vspace{-0.5cm}
\subsection{How to sort your socks while standing on one foot}
You have a lot of socks, and they are all arranged in a line.  We call this a \dfn{sock ordering}. Unfortunately, the colors of the socks are all mixed up. You would like to fix this by sorting the socks into a line so that all socks of the same color are next to each other. Even more unfortunately, the only thing you know how to do with a sock is put it onto or take it off of your foot. For the moment, let us assume that you can put socks onto only one foot. Thus, you must use the following (non-deterministic) \dfn{foot-sorting algorithm} to try sorting your socks. 

You begin by positioning yourself so that your unsorted line of socks is to your right. At each point in time, you can do one of the following operations: 
\begin{itemize}
\item Take the leftmost sock that lies to your right, and put it onto your foot, putting it over all other socks that are already on your foot.
\item Remove the outermost sock from your foot, and place it to your left, placing it to the right of all other socks to your left.
\end{itemize}
\Cref{Fig1} shows a real-life person using this algorithm to sort 7 socks. 
Let us say a sock ordering is \dfn{sorted} if all of the socks of each color appear consecutively.

A sock ordering is \dfn{foot-sortable} if it can be sorted (i.e., transformed into a sorted sock ordering) using this foot-sorting algorithm. We will consider two sock orderings to be the same if one is obtained from the other by renaming the colors. For example, \[\begin{array}{l}\sockb\sockr\sockg\sockr\end{array}\quad\text{and}\quad \begin{array}{l}\sockr\socky\sockg\socky\end{array} \] represent the same sock ordering.  

\subsection{Multiple feet}
Most people have more than one foot; we can naturally generalize the foot-sorting algorithm to take advantage of this fact. Suppose you have $t$ feet arranged in a line (you might need to borrow some friends' feet if $t>2$). As before, position the feet so that the socks are to their right. Let us number the feet $f_1,\ldots,f_t$ from right to left. At each point in time, you can do one of the following operations: 
\begin{itemize}
\item Take the leftmost sock that lies to the right of the feet, and put it onto the foot $f_1$.
\item For some $1\leq i\leq t-1$, remove the outermost sock from the foot $f_i$, and put it on $f_{i+1}$.
\item Remove the outermost sock from the foot $f_t$, and place it to the left of the feet, placing it to the right of all other socks that are to the left of the feet.
\end{itemize}

\begin{figure}[h]
\begin{center}\includegraphics[width=0.216\linewidth]{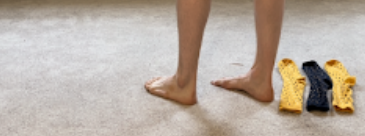}
\includegraphics[width=0.216\linewidth]{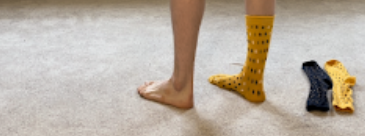}
\includegraphics[width=0.216\linewidth]{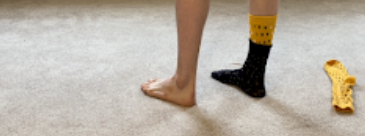}  \vspace{0.2cm}
\includegraphics[width=0.216\linewidth]{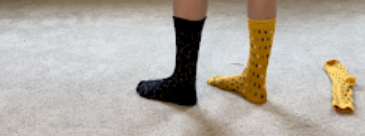} \\
\includegraphics[width=0.216\linewidth]{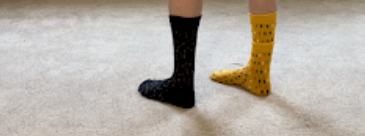}
\includegraphics[width=0.216\linewidth]{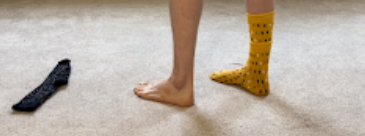}  \vspace{0.2cm}
\includegraphics[width=0.216\linewidth]{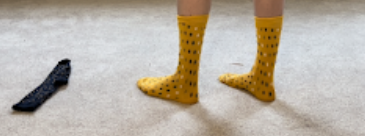}
\includegraphics[width=0.216\linewidth]{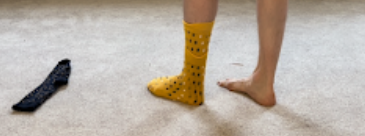} \\
\includegraphics[width=0.216\linewidth]{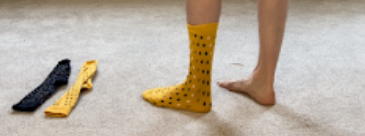}  \vspace{0.2cm}
\includegraphics[width=0.216\linewidth]{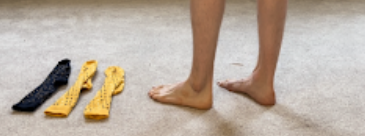}
  \end{center}
  \caption{Three socks begin in the (unsorted) sock ordering \begin{center}\textsf{\textbf{{\color{SockYellow}YELLOW} BLACK {\color{SockYellow}YELLOW}}}.\end{center}  A real person uses the $2$-foot-sorting algorithm to sort these socks. The pictures should be read row by row. }\label{Fig3}
\end{figure}

Let us say a sock ordering is \dfn{$t$-foot-sortable} if it can be sorted using this \dfn{$t$-foot-sorting algorithm} (see \Cref{Fig3}).  It is straightforward to verify that applying the $t$-foot-sorting algorithm is equivalent to applying the (ordinary) foot-sorting algorithm $t$ times since one can use the spaces between adjacent feet as queues. 

\subsection{Set partitions}\label{sec:set-partitions}

A \dfn{set partition} of a set $X$ is a collection $\{B_1,\ldots,B_n\}$ of nonempty pairwise-disjoint subsets of $X$ such that $\bigcup_{i=1}^nB_i=X$. The sets $B_1,\ldots, B_n$ are called the \dfn{blocks} of the partition. Throughout this article, we will always assume that $X$ is a finite set of positive integers. 

Suppose we have $N$ socks, and let $n$ be the total number of different sock colors. We consider socks of the same color to be identical, so we can view a sock ordering as a set partition of the set $[N]:=\{1,\ldots,N\}$ with $n$ blocks: Two numbers $i$ and $j$ lie in the same block of the partition if and only if the socks in positions $i$ and $j$ have the same color. For example, the sock ordering 
\[\begin{array}{l}\sockr\sockb\sockr\sockr\sockg\sockb\end{array}\]
corresponds to the set partition $\{\{{\color{red}1},{\color{red}3},{\color{red}4}\},\{{\color{blue}2},{\color{blue}6}\},\{{\color{MyGreen}5}\}\}$.  By changing the names of the colors, we can also represent this set partition by $abaacb$ (using the color set $\{a,b,c\}$) or $121132$ or $131123$ (using the color set $\{1,2,3\}$).  In what follows, we will tacitly think of set partitions of $[N]$ and sock orderings of length $N$ as the same objects, so it makes sense to discuss \emph{foot-sortable} and, more generally, \emph{$t$-foot-sortable} set partitions.

\subsection{Sock patterns}\label{sec:sock-patterns}
Much research in combinatorics over the last few decades has focused on the notions of \emph{patterns} and \emph{pattern avoidance} for combinatorial objects. There is a natural notion of pattern containment and avoidance for sock orderings. We say the sock ordering $\rho$ \dfn{contains} the sock ordering $\rho'$ if $\rho'$ can be obtained from $\rho$ by deleting some socks. We say $\rho$ \dfn{avoids} $\rho'$ if $\rho$ does not contain $\rho'$. For example, \[\begin{array}{l}\sockr\sockb\sockr\socky\sockg\sockr\end{array}\quad\text{contains}\quad \begin{array}{l}\sockb\sockr\sockg\sockr\end{array}, \] while \[\begin{array}{l}\sockr\sockb\sockr\socky\sockg\sockr\end{array}\quad\hspace{.18cm}\text{avoids}\hspace{.18cm}\quad \begin{array}{l}\sockr\sockb\sockr\sockb\end{array}. \] Let us stress again that sock orderings are considered the same if they correspond to the same set partition. For instance, in the first example, we equally could have said that the sock ordering \[\begin{array}{l}\sockr\sockb\sockr\socky\sockg\sockr\end{array}\quad\text{contains}\quad \begin{array}{l}\sockr\socky\sockg\socky\end{array}. \]

Under the correspondence between sock orderings and set partitions, our notions of containment and avoidance for sock patterns correspond precisely to a notion of containment and avoidance of set partitions that was introduced by Klazar \cite{Klazar1, Klazar2} and investigated further in \cite{Alweiss, Balogh, Bloom, Gunby}. 

A sock ordering is sorted (equivalently, $0$-foot-sortable) if and only if it avoids the pattern $aba$. More generally, for every fixed nonnegative integer $t$, the set of $t$-foot-sortable sock orderings is closed under containment. Indeed, suppose $\rho$ is a $t$-foot-sortable sock ordering that contains the sock ordering $\rho'$. When we apply the $t$-foot-sorting algorithm to sort $\rho$, we perform a sequence of moves that put socks onto feet or remove socks from feet. The moves that involve the socks in $\rho'$ will sort $\rho'$ in the $t$-foot-sorting algorithm.

\subsection{Words and permutations}
A \dfn{word} is a finite sequence of positive integers, possibly with repetitions. The \dfn{standardization} of a word $w$ is the word obtained from $w$ by replacing the $i$-th smallest number in $w$ with $i$ for all $i$. A word is \dfn{standardized} if it equals its standardization. For example, the standardization of $425446$ is $213224$, and the word $213224$ is standardized. The numbers appearing in a word are called \dfn{letters}. A \dfn{permutation} is a word whose letters are all distinct. We write $S_n$ for the set of permutations that use the letters $1,2,\ldots,n$. 

Let $v$ be a standardized word. We call $v$ a \dfn{pattern}, and we say a word $w$ \dfn{contains} $v$ if there is a (not necessarily consecutive) subsequence of $w$ whose standardization is $v$. For example, the word ${\bf 35}65{\bf 1}6$ contains the pattern $231$ because the standardization of the subsequence $351$ is $231$. We say $w$ \dfn{avoids} $v$ if it does not contain $v$. Note that this discussion also defines pattern containment and avoidance for permutations. For example, the permutation $5{\bf 34}62{\bf 1}$ contains the pattern $231$ because the standardization of the subsequence $341$ is $231$; on the other hand, $534621$ avoids the pattern $132$. 

Given a word $w$ of length $N$, there is a set partition $\rho$ of $[N]$ in which each block is the set of positions at which a specific letter occurs in $w$; we say $w$ \dfn{represents} the set partition (equivalently, sock ordering) $\rho$. However, different words can represent the same sock ordering; an example is given by the words $1312$ and $2123$, which both represent \[\begin{array}{l}\sockr\sockb\sockr\sockg\end{array}. \]

\subsection{Stack-sorting}
There is long and rich line of research devoted to sorting procedures that use restricted data structures. One of the first instances of this topic appears in Knuth's book \emph{The Art of Computer Programming} \cite{Knuth2}. Knuth discussed how to use a data structure called a \emph{stack} to sort a permutation. One can push objects into the stack and pop objects out of the stack subject to the constraint that the object being popped out must be the object that was most recently pushed in (see \Cref{Fig1}). This is exactly the same as how we can put socks onto a foot and remove socks from a foot subject to the constraint that the sock being removed from the foot must be the sock that was most recently put onto the foot.\footnote{We could have named the present article ``Stack-Sorting for Set Partitions,'' but we preferred phrasing our sorting procedure in terms of socks and feet because we originally encountered this topic in a real-life attempt to sort socks using a stack (and also because it made the presentation cleaner). Only afterward did it occur to us that the stack should very naturally be replaced by a foot. (What you are reading is a literal footnote.)} 

\begin{figure}[ht]
  \begin{center}{\includegraphics[width=\linewidth]{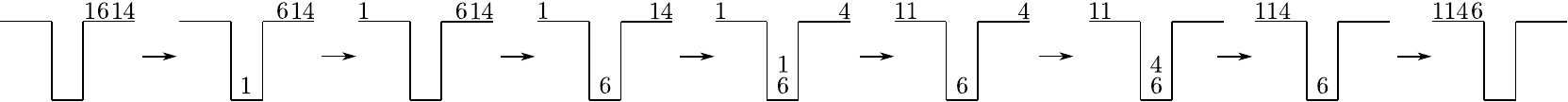}}
  \end{center}
  \caption{Using a stack to sort the word $1614$.}\label{Fig1}
\end{figure}

There are now several well-studied variants of Knuth's original stack-sorting algorithm, including pop-stack-sorting \cite{AtkinsonSack, Avis, SmithVatter}, deterministic (or West) stack-sorting \cite{BonaSurvey, DefantThesis, DefantTroupes, West}, deterministic pop-stack-sorting \cite{ClaessonPop, ClaessonPop2, Pudwell}, (deterministic) left-greedy stack-sorting \cite{Atkinson, Smith}, sorting with pattern-avoiding stacks \cite{Berlow, CerbaiThesis, CerbaiJCTA, DefantZheng}, deterministic stack-sorting for words \cite{DefantKravitzWords}, generalizations to Coxeter groups \cite{DefantCoxeterPop, DefantCoxeterStack}, and so on. As far as we aware, no one has yet studied stack-sorting for set partitions (or socks, for that matter).

\subsection{Main results}

A word is \dfn{sorted} if its letters appear in nondecreasing order.  Note that there are non-sorted words (such as $2211$) that represent sorted sock orderings. The only sorted permutation in $S_n$ is the identity permutation $123\cdots n$. Knuth \cite{Knuth2} characterized the permutations that can be sorted using a single stack via permutation patterns; they are precisely the permutations that avoid that pattern $231$. In fact, this characterization was the original impetus for the now-enormous field of permutation patterns research. In general, the set of permutations that can be sorted by $t$ stacks in series is closed under pattern containment, but when $t\geq 2$, it is very difficult to characterize or enumerate such permutations. Indeed, only recently was it proven that one can decide in polynomial time whether or not a given permutation is 2-stack-sortable \cite{Pierrot}. As shown in \Cref{Fig1}, one can also use a stack to sort words. As in the case of permutations, a word is stack-sortable if and only if it avoids $231$ (the proof is identical to Knuth's proof in the permutation setting). The following theorem relates this fact to sock orderings. 

\begin{theorem}\label{thm:231}
A sock ordering $\rho$ is foot-sortable if and only if there is a $231$-avoiding word that represents $\rho$.
\end{theorem}

In this theorem, it is somewhat unsatisfactory that we have to look at all of the words representing a sock ordering in order to determine whether or not the sock ordering is foot-sortable; it would be more desirable to obtain a characterization purely in terms of set partition patterns.  The latter problem appears to be quite difficult, in analogy with the problem of characterizing (or even enumerating) the $2$-stack-sortable permutations. Although we cannot provide a complete pattern-avoidance characterization of the foot-sortable sock orderings, we can give a very precise description and enumeration when we restrict our attention to sock orderings of a special form.

Let us say a sock ordering is \dfn{$r$-uniform} if there are exactly $r$ socks of each color. The $2$-uniform sock orderings are especially natural examples of sock orderings because socks are often sold and worn in pairs. A set partition corresponding to a $2$-uniform sock ordering is often called a \dfn{matching}. We say a $2$-uniform sock ordering with $2n$ socks is \dfn{alignment-free} if the first $n$ socks all have distinct colors. One can easily check that a $2$-uniform sock ordering is alignment-free if and only if it avoids the pattern\footnote{Our terminology comes from the fact that this pattern corresponds to the set partition $\{\{1,2\},\{3,4\}\}$, which is often called an \dfn{alignment}.} \[\sockr\sockr\sockb\sockb.\] If $\rho$ is an alignment-free $2$-uniform sock ordering with $2n$ socks, then it can be represented by a unique word of the form
$$1 2 \cdots n \sigma_\rho(1) \sigma_\rho(2) \cdots \sigma_\rho(n),$$ where $\sigma_\rho=\sigma_\rho(1) \sigma_\rho(2) \cdots \sigma_\rho(n)$ is a permutation in $S_n$. The map $\rho\mapsto\sigma_\rho$ is a bijection from the set of alignment-free $2$-uniform sock orderings with $2n$ socks to $S_n$. For example, the alignment-free 2-uniform sock ordering \[\begin{array}{l}\sockr\sockb\sockg\socky\sockb\socky\sockg\sockr\end{array} \] corresponds to the permutation $2431$. 

Our next theorem involves the sequence $(F_k)_{k\geq 1}$ of Fibonacci numbers; our convention for indexing this sequence is that $F_1=F_2=1$ and $F_k=F_{k-1}+F_{k-2}$ for all $k\geq 3$. 

\begin{theorem}\label{thm:alignmentless}
Let $\rho$ be an alignment-free $2$-uniform sock ordering of length $2n$, and let \linebreak${\sigma_\rho=\sigma_\rho(1)\cdots\sigma_\rho(n)}$ be the corresponding permutation in $S_n$. Then $\rho$ is foot-sortable if and only if one of the following holds: 
\begin{itemize}
\item $\sigma_\rho(1)<n$, and the permutation $\sigma_\rho(2)\sigma_\rho(3)\cdots\sigma_\rho(n)$ avoids the patterns $123$, $132$, and $213$;
\item $\sigma_\rho(1)=n$, and the permutation $\sigma_\rho(3)\sigma_\rho(4)\cdots\sigma_\rho(n)$ avoids the patterns $123$, $132$, and $213$. 
\end{itemize}
Moreover, there are exactly $(n-1)F_{n+1}$ foot-sortable alignment-free $2$-uniform sock orderings of length $2n$. 
\end{theorem}
We remark that the permutations avoiding the patterns $123$, $132$, and $213$ are precisely the reverses of the layered permutations in which each layer has size $1$ or $2$ (see \Cref{sec:alignmentless}). 

Our last main result states that a person with $n$ different colors of socks can always sort their socks using $\left\lceil\log_2(n)\right\rceil$ feet and that this bound is tight. 

\begin{theorem}\label{thm:log-main}
Every sock ordering with $n$ colors is $\left\lceil\log_2(n)\right\rceil$-foot-sortable. Moreover, there exist sock orderings with $n$ colors that are not $(\left\lceil\log_2(n)\right\rceil-1)$-foot-sortable. 
\end{theorem}

Our proof of the first statement in the preceding theorem uses a simple inductive approach; it is completely analogous to an argument that Tarjan \cite{Tarjan} gave to prove that every permutation of size $n$ is $\left\lceil\log_2(n)\right\rceil$-stack-sortable. On the other hand, our proof of the second statement is more delicate and requires a Ramsey-theoretic argument. 

\subsection{Outline}
We prove \Cref{thm:231,thm:alignmentless,thm:log-main} in \Cref{sec:hare,sec:alignmentless,sec:log}, respectively. In \Cref{sec:further}, we provide a laundry list of natural and intriguing open questions about foot-sorting.

\section{Foot-Sorting and Stack-Sorting}\label{sec:hare}

This brief section is devoted to proving \Cref{thm:231}, which says that a sock ordering is foot-sortable if and only if there is a $231$-avoiding permutation representing it.

\begin{proof}[Proof of \Cref{thm:231}]
Let $\rho$ be a sock ordering with $n$ colors. Suppose first that there is a $231$-avoiding word $w$ over the alphabet $[n]$ that represents $\rho$. Then $w$ is stack-sortable, so there is a sequence of moves that sends $w$ through a stack to transform it into a sorted (i.e., nondecreasing) word. Each of these moves corresponds to a move in the foot-sorting algorithm. If we apply these moves to $\rho$, we will succeed in sorting $\rho$. 

To prove the converse, suppose $\rho$ is foot-sortable. Then there is a sorted sock ordering $\kappa$ that can be obtained by applying the foot-sorting algorithm to $\rho$. Let $v$ be the unique sorted word over the alphabet $[n]$ that represents $\kappa$. Imagine ``reversing time'' and running the foot-sorting algorithm in reverse to transform $\kappa$ into $\rho$. Each move in this ``reversed foot-sorting algorithm'' corresponds to a move in a ``reversed stack-sorting algorithm,'' and applying these moves to $v$ will result in a word $w$ that represents $\rho$. Suppose there exist letters $\alpha<\beta<\gamma$ in $v$ such that in the reversed stack-sorting algorithm, $\alpha$ leaves the stack before both $\beta$ and $\gamma$. Then there is a point in time when $\alpha$ sits above $\beta$ and $\beta$ sits above $\gamma$ in the stack, so $\beta$ must leave the stack before $\gamma$. This means that these letters appear in the order $\gamma,\beta,\alpha$ in $w$. In particular, they can never appear in the order $\beta,\gamma,\alpha$ in $w$. This proves that $w$ avoids $231$. 
\end{proof}

\section{Alignment-Free 2-Uniform Sock Orderings}\label{sec:alignmentless}

In this section, we prove \Cref{thm:alignmentless}, which characterizes and enumerates the foot-sortable alignment-free $2$-uniform sock orderings.  Say that an alignment-free 2-uniform sock ordering is \dfn{spread-out} if it satisfies one of the bullet-pointed conditions in the statement of \Cref{thm:alignmentless}.

The \dfn{skew sum} of two permutations $\sigma\in S_{m}$ and $\sigma'\in S_{m'}$ is the permutation $\sigma\ominus\sigma'\in S_{m+m'}$ defined by 
\[(\sigma\ominus\sigma')(i)=\begin{cases} \sigma(i)+m' & \mbox{if } 1\leq i\leq m; \\ \sigma'(i-m) & \mbox{if } m+1\leq i\leq m+m'. \end{cases}\]
Note that the skew sum is an associative operation. It is well known that a permutation avoids the patterns $123$, $132$, and $213$ if and only if it is of the form $\varepsilon_1\ominus\varepsilon_2\ominus\cdots\ominus\varepsilon_k$,
where each $\varepsilon_i$ is either $1$ or $12$  \cite{Bevan, Linton}. (Said differently, a permutation avoids $123$, $132$, and $213$ if and only if it is the reverse of a \emph{layered} permutation whose layers all have size $1$ or $2$.)  It will be convenient to substitute this characterization into the bullet-pointed descriptions in the statement of \Cref{thm:alignmentless}.

We will first show that foot-sortable alignment-free $2$-uniform sock orderings are spread-out.  Let $\rho$ be a foot-sortable alignment-free $2$-uniform sock ordering, and let $\sigma_\rho \in S_n$ denote the corresponding permutation.  There is nothing to show if $n \leq 2$, so assume that $n \geq 3$. For $\tau\in S_n$, write $\tau\sigma_\rho$ for the product of $\tau$ and $\sigma_\rho$ in the symmetric group $S_n$. \Cref{thm:231} tells us that there exists $\tau \in S_n$ such that the word $w$ obtained by concatenating $\tau$ with $\tau \sigma_\rho$ represents $\rho$ and avoids the pattern $231$.  Note that the numbers $2, \ldots, n$ appear in decreasing order in $\tau$ since otherwise we would find a $231$ pattern in $w$.  Hence, there is some $1 \leq i \leq n$ such that $\tau$ has the one-line notation
$$(n)(n-1)\cdots(n-i+2)(1)(n-i+1)\cdots (2)$$
and $\tau^{-1}$ has the one-line notation
$$(i)(n)(n-1)\cdots (i+1)(i-1)(i-2)\cdots (1).$$
Since $w$ avoids $231$, we also see that $\tau \sigma_\rho$ is a layered permutation whose layers all have size $1$ or $2$.  Now we use the formula $\sigma_\rho=\tau^{-1}(\tau \sigma_\rho)$ to show that $\sigma_\rho$ has the desired form.  We condition on the size of the first run in $\tau \sigma_\rho$.  First, suppose that the first run in $\tau \sigma_\rho$ has size $1$.  Then $\tau \sigma_\rho(1)=1$, and $\sigma_\rho(1)=i$.  In this case, one easily verifies that the standardization of $\sigma_\rho(2) \cdots \sigma_\rho(n)$ is the skew sum of some number of increasing permutations of sizes $1$ and $2$; thus, regardless of whether or not $\sigma_\rho(1)=1$, we conclude that $w$ satisfies one of the bullet-pointed conditions in the statement of \Cref{thm:alignmentless} and hence is spread-out.  Next, suppose that the first run in $\tau \sigma_\rho$ has size $2$.  Then we have $\sigma_\rho(1)=\tau^{-1}(2)=n$ and $\sigma_\rho(2)=\tau^{-1}(1)=i$, and (as in the previous case) the standardization of $\sigma_\rho(3) \cdots \sigma_\rho(n)$ is the skew sum of some number of increasing permutations of sizes $1$ and $2$.  So $w$ satisfies the second bullet-pointed condition in the statement of \Cref{thm:alignmentless} and hence is spread-out.  This completes the proof that foot-sortable alignment-free $2$-uniform sock orderings are spread-out.

We will now show that spread-out sock orderings are foot-sortable.  Let $\rho$ be a spread-out sock ordering, and let $\sigma_\rho \in S_n$ denote the corresponding permutation.  First, suppose that $\sigma_\rho$ satisfies the first bullet-pointed condition in \Cref{thm:alignmentless}.  Let $i=\sigma_\rho(1)$, and let $\tau \in S_n$ be the permutation with the one-line notation
$$(n)(n-1)\cdots(n-i+2)(1)(n-i+1)\cdots (2).$$
The word obtained by concatenating $\tau$ with $\tau \sigma_\rho$ represents $\rho$ and avoids the pattern $231$, so $\rho$ is foot-sortable by \Cref{thm:231}.
Next, suppose that $\sigma_\rho$ satisfies the second bullet-pointed condition in \Cref{thm:alignmentless}.  Let $i=\sigma_\rho(2)$, and again let $\tau \in S_n$ be the permutation with the one-line notation
$$(n)(n-1)\cdots(n-i+2)(1)(n-i+1)\cdots (2).$$
The word obtained by concatenating $\tau$ with $\tau \sigma_\rho$ represents $\rho$ and avoids the pattern $231$, so $\rho$ is foot-sortable by \Cref{thm:231}.  This completes the proof of the first statement of \Cref{thm:alignmentless}; to prove the second statement, it remains to count the spread-out alignment-free 2-uniform sock orderings.

Let $\Av_m(123,132,213)$ denote the set of permutations in $S_m$ that avoid the patterns $123, 132, 213$. It is well known that $|\Av_m(123,132,213)|$ is the Fibonacci number $F_{m+1}$.  The permutations $\sigma$ corresponding to the spread-out alignment-free 2-uniform sock orderings with $n$ colors can be described as follows, according to the two bullet points in the statement of \Cref{thm:alignmentless}:
\begin{itemize}
    \item We can choose any element of $[n-1]$ to be $\sigma(1)$, and then we can choose any element of $\Av_{n-1}(123,132,213)$ to be the standardization of $\sigma(2) \cdots \sigma(n)$.
    \item We can put $\sigma(1)=n$, choose any element of $[n-1]$ to be $\sigma(2)$, and choose any element of $\Av_{n-2}(123,132,213)$ to be the standardization of $\sigma(3) \cdots \sigma(n)$.
\end{itemize}
Summing the contributions of these two possibilities, we find that the total number of spread-out alignment-free 2-uniform sock orderings with $n$ colors is
$$(n-1)|\Av_{n-1}(123,132,213)|+(n-1)|\Av_{n-2}(123,132,213)|=(n-1)F_n+(n-1)F_{n-1}=(n-1)F_{n+1},$$
as desired.

\begin{remark}
One can show that an alignment-free 2-uniform sock ordering is spread-out if and only if it avoids all of the following sock orderings as patterns:
\begin{align*}
&\ssockr\ssockb\ssockg\ssocky\ssockp\ssockp\ssocky\ssockr\ssockb\ssockg,\quad \ssockr\ssockb\ssockg\ssocky\ssockp\ssocky\ssockp\ssockr\ssockb\ssockg, \\ &\ssockr\ssockb\ssockg\ssocky\ssockp\ssockp\ssocky\ssockr\ssockg\ssockb, \quad \ssockr\ssockb\ssockg\ssocky\ssockp\ssocky\ssockp\ssockr\ssockg\ssockb, \\ 
&\ssockr\ssockb\ssockg\ssocky\ssockp\ssockp\ssocky\ssockb\ssockr\ssockg, \quad \ssockr\ssockb\ssockg\ssocky\ssockp\ssocky\ssockp\ssockb\ssockr\ssockg,
\end{align*}
\begin{align*}
&\ssockr\ssockb\ssockg\ssocky\ssockr\ssockb\ssockg\ssocky,\quad \ssockr\ssockb\ssockg\ssocky\ssockb\ssockg\ssockr\ssocky,\quad \ssockr\ssockb\ssockg\ssocky\ssockr\ssockb\ssocky\ssockg, \\ 
&\ssockr\ssockb\ssockg\ssocky\ssockg\ssockr\ssockb\ssocky, \quad \ssockr\ssockb\ssockg\ssocky\ssockr\ssockg\ssockb\ssocky, \quad \ssockr\ssockb\ssockg\ssocky\ssockg\ssockr\ssocky\ssockb, \\
&\ssockr\ssockb\ssockg\ssocky\ssockb\ssockr\ssockg\ssocky, \quad \ssockr\ssockb\ssockg\ssocky\ssockg\ssockb\ssockr\ssocky, \quad \ssockr\ssockb\ssockg\ssocky\ssockb\ssockr\ssocky\ssockg. 
\end{align*}
(These patterns are the alignment-free 2-uniform sock orderings corresponding to the permutations $54123, 45123, 54132, 45132, 54213, 45213, 1234, 2314, 1243, 3124, 1324, 3142, 2134, 3214,$ $2143$.)
Using a computer (or some socks and a sufficient amount of patience), one can check that none of the sock orderings in this list are foot-sortable; note that this gives an alternative approach to the first half of the above proof of \Cref{thm:alignmentless}.
\end{remark}

\section{Logarithmically Many Feet}\label{sec:log}

The goal of this section is to answer the following question: For a natural number $n$, what is the minimum value of $t$ (depending on $n$) such that every sock ordering with $n$ colors is $t$-foot-sortable?  We will  prove \Cref{thm:log-main}, which states that the answer is exactly $\lceil \log_2(n) \rceil$.  

Although the motivation for defining the $t$-foot-sorting algorithm is to sort a sock ordering, one can use it to permute the socks in other ways as well. For example, one could use the foot-sorting algorithm to transform \[\begin{array}{l}\sockr\sockb\sockr\sockb\sockg\end{array}\quad\text{into}\quad\begin{array}{l}\sockb\sockr\sockr\sockg\sockb\end{array}.\] Given a sock ordering $\rho$, let $\foot(\rho)$ denote the set of sock orderings that can be obtained from $\rho$ by applying the foot-sorting algorithm with a single foot, and let $\foot^{-1}(\rho)$ denote the set of sock orderings $\kappa$ such that $\rho\in\foot(\kappa)$.  We define $\foot^0(\rho)=\{\rho\}$.  For $t>1$, we inductively define $\foot^t(\rho)=\foot(\foot^{t-1}(\rho))$ and $\foot^{-t}(\rho)=\foot^{-1}(\foot^{-t+1}(\rho))$; it is a simple exercise to verify that $\foot^t(\rho)$ is the set of sock orderings that can be obtained from $\rho$ by applying the $t$-foot-sorting algorithm. 

We can quickly prove the first statement in \Cref{thm:log-main}, which states that $\lceil \log_2(n) \rceil$ feet always suffice to sort a sock ordering with $n$ colors. This argument is very similar to one used by Tarjan in \cite{Tarjan}. 

\begin{lemma}\label{lem:log_first_part}
Every sock ordering with $n$ colors is $\left\lceil\log_2(n)\right\rceil$-foot-sortable. 
\end{lemma}

\begin{proof}
It suffices to prove the theorem when $n$ is a power of $2$; say $n=2^k$.  We proceed by induction on $k$.  The base case $k=0$ is trivial since any sock ordering with a single color is already sorted.

Now suppose $k\geq 1$, and assume that all sock orderings with $2^{k-1}$ colors are $(k-1)$-foot-sortable. Let $\rho$ be a sock ordering with $2^{k}$ colors.  Partition the set of colors into two sets $A,B$, each of size $2^{k-1}$.  Apply the foot-sorting algorithm once to $\rho$ so that in the resulting sock ordering $\rho' \in \foot(\rho)$, all of the socks with colors from $A$ appear to the left of all of the socks with colors from $B$.  (For instance, sort according to the rules that every sock with color from $A$ is removed from the foot immediately after being put on the foot and that no sock with color from $B$ is removed from the foot until all socks with colors from $A$ have been removed.)  Write $\rho'$ as a concatenation $\rho'=\kappa_A \kappa_B$, where $\kappa_A$ and $\kappa_B$ contain only socks with colors from $A$ and from $B$, respectively.  By the induction hypothesis, each of $\kappa_A$ and $\kappa_B$ can be sorted using $k-1$ further applications of foot-sorting.  Since these further applications can be run in parallel, we conclude that $\rho$ is $k$-foot-sortable.
\end{proof}

It is much more difficult to show that some sock orderings with $n$ colors require $\lceil \log_2(n) \rceil$ feet to get sorted. To do so, we will study a particular class of sock orderings and show that they obey a Ramsey-theoretic property.  Say that an $r$-uniform sock ordering with $n$ colors is \dfn{stratified} if it is the concatenation of $r$ chunks of size $n$, where each color appears once in each chunk.  For example, the $2$-uniform stratified sock orderings with $3$ colors are 
\[\sockr\sockb\sockg\sockr\sockb\sockg,\quad \sockr\sockb\sockg\sockr\sockg\sockb,\quad \sockr\sockb\sockg\sockb\sockr\sockg,\] \[ \sockr\sockb\sockg\sockb\sockg\sockr,\quad \sockr\sockb\sockg\sockg\sockr\sockb,\quad \sockr\sockb\sockg\sockg\sockb\sockr.\] The following lemma says that if you apply the foot-sorting algorithm to a stratified sock ordering with $n$ colors that has huge uniformity, then you are always guaranteed to find (as a pattern) a stratified sock ordering with $\lceil n/2 \rceil$ colors that has pretty big uniformity.

\begin{lemma}\label{lem:ramsey}
Let $n \geq 2$ and $r$ be positive integers, and set $r'=r^2\binom{n}{\left\lceil n/2\right\rceil}$.  If $\rho$ is an $r'$-uniform stratified sock ordering with $n$ colors, then every sock ordering in $\foot(\rho)$ contains an $r$-uniform stratified sock ordering with $\lceil n/2 \rceil$ colors as a pattern.
\end{lemma}

\begin{proof}
Set $m=r'/r=r\binom{n}{\left\lceil n/2\right\rceil}$, and write $\rho$ as a concatenation $$\rho=\gamma_1 \gamma_2 \cdots \gamma_m,$$ where each $\gamma_k$ has $nr$ socks.  (So each $\gamma_k$ is an $r$-uniform stratified sock ordering with $n$ colors.)  Imagine applying the foot-sorting algorithm to transform $\rho$ into some sock ordering $\kappa\in\foot(\rho)$. For each $k\in[m]$, let $X_k$ be the set of socks in $\gamma_k$ that are removed from the foot before the last sock in $\gamma_k$ is put onto the foot. Let $Y_k$ be the set of distinct colors appearing in $X_k$. We consider two cases. 

\medskip
\noindent {\bf Case 1.} Suppose there exists $k\in [m]$ such that $|Y_k|\leq \left\lfloor n/2\right\rfloor$. Let $Q$ be a subset of $[n]\setminus Y_k$ of size $\lceil n/2 \rceil$. At the moment when the last sock in $\gamma_k$ is put on the foot, the socks of $\gamma_k$ with colors in $Q$ are all on the foot, and they appear in the same order (read from inside to outside) in which they appeared in $\rho$; note that these socks in $\rho$ form a pattern that is an $r$-uniform stratified sock ordering with $\lceil n/2 \rceil$ colors.  After the entire foot-sorting algorithm is complete, these same socks appear in $\kappa$ in the opposite order from how they appeared in $\rho$.  Since the reversal of a stratified sock ordering is still stratified, we have succeeded in finding the desired pattern in $\kappa$.

\medskip
\noindent {\bf Case 2.} Suppose $|Y_k|\geq\left\lceil n/2\right\rceil$ for every $k\in[m]$. Since $m=r\binom{n}{\left\lceil n/2\right\rceil}$, we can use the Pigeonhole Principle to find a subset $Q$ of $[n]$ of size $\lceil n/2 \rceil$ such that there are at least $r$ indices $k\in [m]$ with $Q\subseteq Y_k$; let $k_1<\cdots<k_r$ denote $r$ such indices.  For each $k_i$, choose a subset $A_i \subseteq X_{k_i}$ of size $\lceil n/2 \rceil$ such that $A_i$ has one sock of each color in $Q$.  The definition of $X_{k_1}$ ensures that all of the socks in $A_1$ are removed from the foot before the last sock of $\gamma_{k_1}$ is put on the foot.  After this, but before the last sock of $\gamma_{k_2}$ is put on the foot, all of the socks in $A_2$ are removed from the foot.  Continuing in this fashion, we find a sock pattern in $\kappa$ consisting of some permutation of the socks in $A_1$, followed by some permutation of the socks in $A_2$, and so on, up to some permutation of the socks in $A_r$.  But this is precisely an occurrence of an $r$-uniform stratified sock ordering with $\lceil n/2 \rceil$ colors, as desired.
\end{proof}

We are now ready to prove \Cref{thm:log-main}, which says that every sock ordering with $n$ colors is $\lceil \log_2(n) \rceil$-foot-sortable and that sometimes $\lceil \log_2(n) \rceil$ feet are required.

\begin{proof}[Proof of \Cref{thm:log-main}]
The first statement in \Cref{thm:log-main} is \Cref{lem:log_first_part}. To prove the second statement, let $r(2)=2$, and let \[r(n)=r(\left\lceil n/2\right\rceil)^2\binom{n}{\lceil n/2\rceil}\] for all $n\geq 3$. We will prove that for every $n\geq 2$, no $r(n)$-uniform stratified sock ordering with $n$ colors is $(\left\lceil\log_2(n)\right\rceil-1)$-foot-sortable. This is easy when $n=2$ because the two different $2$-uniform stratified sock orderings with 2 colors are $abab$ and $abba$, neither of which is sorted. We now assume $n\geq 3$ and proceed by induction on $n$. 

Let $\rho$ be an $r(n)$-uniform stratified sock ordering with $n$ colors. Our goal is to show that no sock ordering in $\foot(\rho)$ is $(\left\lceil\log_2(n)\right\rceil-2)$-foot-sortable. Choose any $\kappa\in\foot(\rho)$. According to \Cref{lem:ramsey}, the sock ordering $\kappa$ contains an $r(\lceil n/2\rceil)$-uniform stratified sock ordering $\eta$ with $\lceil n/2\rceil$ colors as a pattern. By induction, $\eta$ is not $(\left\lceil\log_2(\left\lceil n/2\right\rceil)\right\rceil-1)$-foot-sortable. This implies that $\kappa$ is also not $(\left\lceil\log_2(\left\lceil n/2\right\rceil)\right\rceil-1)$-foot-sortable, and the theorem now follows from the simple fact that $\left\lceil\log_2(\left\lceil n/2\right\rceil)\right\rceil=\left\lceil\log_2(n)\right\rceil-1$. 
\end{proof}

One can show that the function $r(n)$ defined in the preceding proof satisfies $r(n)= n^{(1+o(1))n}$.  We conclude this section by showing the existence of sock orderings with less enormous uniformity that still require many feet to get sorted.  Let $C_N=\frac{1}{N+1}\binom{2N}{N}$ denote the $N$-th Catalan number. We will make use of the well-known estimate 
\begin{equation}\label{eq:4n}
C_N\leq 4^N.
\end{equation}

\begin{lemma}\label{prop:catalan}
If $\rho$ is a sock ordering of length $N$, then $|\foot^{-1}(\rho)|\leq C_N$.
\end{lemma}

\begin{proof}
Choose $\kappa \in \foot^{-1}(\rho)$.  Consider some way of running the foot-sorting algorithm to obtain $\rho$ from $\kappa$.  Record the letter U each time you put a sock onto the foot, and record the letter D each time you take a sock off of the foot. This yields a sequence $\text{Dyck}_\rho(\kappa)$ containing $N$ U's and $N$ D's with the property that each initial subsequence contains at least at many U's as D's. Such a sequence is called a \dfn{Dyck word}. For example, the Dyck word corresponding to the application of the foot-sorting algorithm in \Cref{Fig2} is UUDUUDUDDUUDDD. The resulting map $\kappa \mapsto \text{Dyck}_\rho(\kappa)$ is an injection from $\foot^{-1}(\rho)$ to the set of Dyck words of length $2N$. It is well known that the number of Dyck words of length $2N$ is $C_N$, so this completes the proof. 
\end{proof}

\begin{proposition}\label{prop:weak-log-bound}
Let $n,r \geq 2$ be integers.  There exists an $r$-uniform sock ordering with $n$ colors that is not $(\left\lfloor\frac{r-1}{r}\log_4(n)\right\rfloor-1)$-foot-sortable.
\end{proposition}

\begin{proof}
Let $U_{n,r}$ denote the set of $r$-uniform sock orderings with $n$ colors, and let $\kappa_0$ be the unique sorted sock ordering in 
 $U_{n,r}$.  Induction on $k$ and iterative applications of \Cref{prop:catalan} yield the inequality
$$|\foot^{-k}(\kappa_0)| \leq C_{nr}^k$$
for all positive integers $k$.  Note that $\foot^{-k}(\kappa_0)$ is precisely the set of $k$-foot-sortable elements of $U_{n,r}$.  If $|\foot^{-k}(\kappa_0)|<|U_{n,r}|$, then we can conclude that there exists a sock ordering in $U_{n,r}$ that is not $k$-foot sortable.  

Using the well-known inequalities $m^me^{-m+1}<m!<m^{m+1}e^{-m+1}$, we find that 
\[|U_{n,r}|=\frac{(nr)!}{n!(r!)^n} \geq\frac{(nr)^{nr}e^{-nr+1}}{n^{n+1}e^{-n+1}(r^{r+1}e^{-r+1})^n} =n^{nr-n-1}r^{-n}.\] We also know by \eqref{eq:4n} that $C_{nr}^k<4^{knr}$. It is a simple exercise to check that ${4^{knr}<n^{nr-n-1}r^{-n}}$ when $k=\left\lfloor\frac{r-1}{r}\log_4(n)\right\rfloor-1$.
\end{proof}

When $r=2$, \Cref{thm:log-main,prop:weak-log-bound} tell us that the minimum number of feet needed to sort every sock ordering with $n$ colors is between $\left\lfloor\frac{1}{2}\log_4(n)\right\rfloor$ and $\left\lceil\log_2(n)\right\rceil$; when $n$ is large, these bounds differ by a constant factor of around $4$. When $r$ is much larger than $2$, the fraction $\frac{r-1}{r}$ is close to $1$, so the bounds $\left\lfloor\frac{r-1}{r}\log_4(n)\right\rfloor$ and $\left\lceil\log_2(n)\right\rceil$ differ by a constant factor of around $2$.

\section{Further Directions}\label{sec:further}
The purpose of this final section is to raise several questions and directions that seem promising for future inquiry into foot-sorting; we have barely scratched the surface in this paper.  

Recall that we have defined notions of pattern containment and avoidance for sock orderings that correspond precisely to Klazar's notions of pattern containment and avoidance for set partitions. If $\Omega$ is a set of sock orderings that is closed under pattern containment, then the \dfn{basis} of $\Omega$ is the set $B$ of minimal (under the pattern containment partial order) sock orderings that are not in $\Omega$. One can characterize $\Omega$ as the set of sock orderings that avoid all of the patterns in $B$. Recall that the set of foot-sortable sock orderings is closed under pattern containment. 

\begin{question}\label{ques:avoiding}
Is the basis of the set of foot-sortable sock orderings finite?  How about the basis of the set of foot-sortable sock orderings in which each color is used at most twice?
\end{question}

\begin{remark}
In the time since the original preprint of this article was released, Yu has resolved \Cref{ques:avoiding} by showing that both of the bases considered in that question are infinite \cite{YuPrivate}. Yu has also provided a polynomial-time algorithm for deciding whether a given sock ordering is foot-sortable.
\end{remark}

It would also be interesting to enumerate the $2$-uniform foot-sortable sock orderings with $n$ colors.

Another line of investigation concerns bounds for ``worst-case sorting,'' as studied in \Cref{sec:log}. 

\begin{question}\label{Quest:t_2}
What is the smallest integer $t_2(n)$ such that every $2$-uniform sock ordering with $n$ colors is $t_2(n)$-foot-sortable?
\end{question}

Recall that, as discussed after the proof of \Cref{prop:weak-log-bound}, we have \[\left\lfloor\textstyle{\frac{1}{2}}\log_4(n)\right\rfloor\leq t_2(n)\leq\left\lceil \log_2(n) \right\rceil.\]  

Define ${\bf r}(n)$ to be the smallest integer such that there exists an ${\bf r}(n)$-uniform sock ordering with $n$ colors that is not $(\left\lceil\log_2(n)\right\rceil-1)$-foot-sortable. \Cref{thm:log-main} states that ${\bf r}(n)$ is finite, and it follows from our proof of this theorem that ${\bf r}(n)\leq n^{(1+o(1))n}$. This function remains mysterious, and it is not even obvious whether or not ${\bf r}(n)$ tends to infinity with $n$. 
\begin{question}
What is ${\bf r}(n)$? Is it true that $\lim\limits_{n\to\infty}{\bf r}(n)=\infty$? 
\end{question}

Avis and Newborn \cite{Avis} introduced a variant of Knuth's stack data structure called a \dfn{pop-stack} (see also \cite{AtkinsonSack,SmithVatter}). A pop-stack is just like a stack, except it has the additional property that whenever an object is removed from the pop-stack, \emph{all} objects must be removed. One could analogously consider sorting socks using \dfn{pop-feet}, which are just like feet except that whenever a sock is removed from a pop-foot, \emph{all} socks must be removed from the pop-foot. It would be interesting to study $t$-pop-foot sortable sock orderings.

In the stack-sorting and pop-stack sorting literature, one can consider sorting permutations using stacks or pop-stacks that are arranged in series or in parallel (see \cite{AlbertBousquet,AtkinsonSack,SmithVatter}). Our $t$-foot-sorting algorithm uses $t$ feet in series, but one could also consider what happens when the feet are arranged in parallel. In this setting, we imagine $t$ feet arranged in a line as before. At each point in time, the leftmost sock lying to the right of the feet can be placed directly onto one of the feet (and any of the feet can be chosen), or a sock can be removed from one of the feet and put into the output sock ordering. It could be interesting to study which sock orderings can be sorted using $t$ feet in parallel. One could also consider mixtures of feet and pop-feet in series and in parallel, as in \cite{SmithVatter}.  

Finally, we mention that Xia \cite{Xia} has introduced deterministic versions of foot-sorting (in which stacks are used in place of feet).

\section*{Acknowledgments}
We thank Anshul Adve for providing the initial real-world motivation for the investigations in this paper.  We are also grateful to Noga Alon, Alex Postnikov, and Vince Vatter for helpful discussions. We thank the anonymous referees for providing several helpful comments, including suggestions for how to simplify the proof of \cref{thm:alignmentless}. Colin Defant was supported by the National Science Foundation under Award No. 2201907 and by a Benjamin Peirce Fellowship at Harvard University, and Noah Kravitz was supported by the NSF Graduate Research Fellowship Program under grant DGE--2039656.


\begin{thebibliography}{99}

\bibitem{AlbertBousquet}
M. Albert and M. Bousquet-M\'elou, Permutations sortable by two stacks in parallel and quarter-plane walks. \emph{European J. Combin.}, {\bf 43} (2015), 131--164. 

\bibitem{Alweiss}
R. Alweiss, Asymptotic results on Klazar set partition avoidance. \emph{Discrete Math. Theoret. Comput. Sci.}, {\bf 19} (2018). 

\bibitem{Atkinson}
M. D. Atkinson, M. M. Murphy, and N. Ru\v{s}kuc, Sorting with two ordered stacks in series.
\emph{Theoret. Comput. Sci.} {\bf 289} (2002), 205--223.

\bibitem{AtkinsonSack}
M. D. Atkinson and J.-R. Sack. Pop-stacks in parallel. \emph{Inform. Process. Lett.},
{\bf 70} (1999), 63--67.

\bibitem{Avis}
D. Avis and M. Newborn, On pop-stacks in series. \emph{Utilitas Math.}, {\bf 19} (1981),
129--140.

\bibitem{Balogh}
J. Balogh, B. Bollob\'as, and R. Morris, Hereditary properties of partitions,
ordered graphs and ordered hypergraphs. \emph{European J. Combin.}, {\bf 27} (2006), 1263--1281.

\bibitem{Berlow}
K. Berlow, Restricted stacks as functions. \emph{Discrete Math.}, {\bf 344} (2021).

\bibitem{Bevan}
D. Bevan, Permutation patterns: basic definitions and notation. \emph{Preprint} arXiv:1506.06673v1 (2015). 

\bibitem{Bloom}
J. Bloom and D. Saracino, Pattern avoidance for set partitions \`a la Klazar. \emph{Discrete Math. Theoret. Comput. Sci.}, {\bf 18} (2016). 

\bibitem{BonaSurvey}
M. B\'ona, A survey of stack-sorting disciplines. \emph{Electron. J. Combin.}, {\bf 9} (2003). 

\bibitem{CerbaiThesis}
G. Cerbai, Sorting permutations with pattern-avoiding machines. Ph.D. Thesis, University of Firenze (2021). 

\bibitem{CerbaiJCTA}
G. Cerbai, A. Claesson, and L. Ferrari, Stack sorting with restricted stacks. \emph{J. Combin. Theory Ser. A}, {\bf 173} (2020).

\bibitem{ClaessonPop}
A. Claesson and B. \'A. Gu{\dh}mundsson, Enumerating permutations sortable by $k$ passes through a pop-stack.\ \emph{Adv. Appl. Math.}, {\bf 108} (2019), 79--96. 

\bibitem{ClaessonPop2}
A. Claesson, B. \'A. Gu{\dh}mundsson, and J. Pantone, Counting pop-stacked permutations in polynomial time.  \emph{Experiment. Math.} (to appear).  

\bibitem{DefantThesis}
C. Defant, Stack-sorting and beyond. Ph.D. Thesis, Princeton University (2022). 

\bibitem{DefantCoxeterPop}
C. Defant, Pop-stack-sorting for Coxeter groups.	\emph{Comb. Theory}, {\bf 2} (2022).

\bibitem{DefantCoxeterStack}
C. Defant, Stack-sorting for Coxeter groups.	\emph{Comb. Theory}, {\bf 2} (2022).

\bibitem{DefantTroupes}
C. Defant, Troupes, cumulants, and stack-sorting. \emph{Adv. Math.}, {\bf 399} (2022). 

\bibitem{DefantKravitzWords}
C. Defant and N. Kravitz, Stack-sorting for words. \emph{Australas. J. Combin.}, {\bf 77} (2020), 51--68. 

\bibitem{DefantZheng}
C. Defant and K. Zheng, Stack-sorting with consecutive-pattern-avoiding stacks. \emph{Adv. Appl.\ Math.}, {\bf 128} (2021). 

\bibitem{Gunby}
B. Gunby, Asymptotics of pattern avoidance in the Klazar set partition and permutation-tuple settings. \emph{European J. Combin.}, {\bf 82} (2019). 

\bibitem{Klazar1}
M. Klazar, Counting pattern-free set partitions I: a Generalization of Stirling numbers of the second kind. \emph{European J. Combin.}, {\bf 21} (2000), 367--378. 

\bibitem{Klazar2}
M. Klazar, Counting pattern-free set partitions II: noncrossing and other hypergraphs. \emph{Electron. J. Combin.}, {\bf 7} (2000).

\bibitem{Knuth2}
D. Knuth, \emph{The Art of Computer Programming, Volume III: Sorting and Searching}. Addison-Wesley, 1973.

\bibitem{Linton}
S. Linton, N. Ru\v{s}kuc, and V. Vatter, \emph{Permutation patterns}.  London Mathematical Society Lecture Note Series, vol. 376, Cambridge University Press, 2010.

\bibitem{Pierrot}
A. Pierrot and D. Rossin, 2-stack-sorting is polynomial. \emph{Theory Comput. Syst.}, {\bf 60} (2017), 552--579. 

\bibitem{Pudwell}
L. Pudwell and R. Smith, Two-stack-sorting with pop stacks.\ \emph{Australas. J. Combin.}, {\bf 74} (2019), 179--195. 

\bibitem{Smith}
R. Smith, Comparing algorithms for sorting with $t$ stacks in series. \emph{Ann. Comb.}, {\bf 8} (2004), 113--121. 

\bibitem{SmithVatter}
R. Smith and V. Vatter, A stack and pop stack in series. \emph{Australas. J. Combin.}, {\bf 58} (2014), 157--171. 

\bibitem{Tarjan}
R. Tarjan, Sorting using networks of queues and stacks. \emph{J. ACM}, {\bf 19} (1972), 341--346. 

\bibitem{West}
J. West, Permutations with restricted subsequences and stack-sortable permutations, Ph.D. Thesis, MIT (1990). 

\bibitem{Xia}
J. Xia, Deterministic stack-sorting for set partitions. \emph{Preprint} arXiv:2309.14644 (2023).

\bibitem{YuPrivate}
H.-H. H. Yu, Deciding foot-sortability and minimal 2-bounded non-foot-sortable sock orderings. \emph{Preprint} arXiv:2312.14397 (2023).

\end{thebibliography}
\end{document}